\documentclass{amsart}
\usepackage{amsmath,amssymb,amsthm,xcolor}

\newtheorem{theorem}[equation]{Theorem}
\newtheorem{lemma}[equation]{Lemma}

\newtheorem{proposition}[equation]{Proposition}

\newcommand{\Mod}[1]{\ (\mathrm{mod}\ #1)}

\def\F{\mathbb{F}}

\begin{document}

\title{Algebraic Properties of a Hypergraph Lifting Map}

\author[M. Budden]{Mark Budden}
\address{Department of Mathematics and Computer Science \\
Western Carolina University \\
Cullowhee, NC 28723 USA}
\email{mrbudden@email.wcu.edu}

\author[J. Hiller]{Josh Hiller}
\address{Department of Mathematics and Computer Science \\ 
Adelphi University \\
Garden City, NY 11530-0701}
\email{johiller@adelphi.edu}

\author[T. Meek]{Tommy Meek}
\address{Department of Mathematics and Computer Science \\
Western Carolina University \\
Cullowhee, NC 28723 USA}
\email{thmeek1@catamount.wcu.edu}

\author[A. Penland]{Andrew Penland}
\address{Department of Mathematics and Computer Science \\
Western Carolina University \\
Cullowhee, NC 28723 USA}
\email{adpenland@email.wcu.edu}


\begin{abstract}
Recent work in hypergraph Ramsey theory has involved the introduction of a ``lifting map'' that associates a certain $3$-uniform hypergraph to a given graph, bounding cliques in a predictable way.  In this paper, we interpret the lifting map as a linear transformation. This interpretation allows us to use algebraic techniques to prove several structural properties of the lifting map, culminating in new lower bounds for certain $3$-uniform hypergraph Ramsey numbers. \end{abstract}

\maketitle

\section{Introduction}

In \cite{BHLS}, a lifting map $\varphi :\mathcal{G}_2\longrightarrow \mathcal{G}_3$ was described that assigned to each graph a unique $3$-uniform hypergraph.  Here, $\mathcal{G}_2$ denotes the set of all graphs of order at least $3$ and $\mathcal{G}_3$ is the set of all $3$-uniform hypergraphs of order at least $3$.  Both a graph $G$ and its image $\varphi (G)$ share the same vertex set, and an unordered $3$-tuple $abc$ forms a hyperedge in $\varphi (G)$ if and only if the subgraph of $G$ induced by $\{a, b, c\}$ contains an odd number of edges. 

The lifting $\varphi$ was shown to preserve complements (i.e., $\varphi (\overline{G})=\overline{\varphi (G)}$) and the way in which $\varphi$ lifted to complete graphs was analyzed.  Specifically, it was shown that if $\varphi (G)$ contained a complete hypergraph with vertices $x_1, x_2, \dots , x_n$, then the subgraph of $G$ induced by $\{ x_1, x_2, \dots , x_n\}$ is the disjoint union of at most $2$ complete graphs (including the possibility that it is complete).  These properties were then used to provide new lower bounds for certain $3$-uniform hypergraph Ramsey numbers.  

In \cite{BR}, a generalization of $\varphi$ was described that allowed graphs to be lifted to $r$-uniform hypergraphs.   In this variation, denoted $\varphi ^{(r)}$, a hyperedge $x_1x_2\cdots x_r$ is formed in the image of $\varphi ^{(r)}$ if and only if the subgraph of $G$ induced by $\{x_1, x_2, \dots , x_r\}$ is the disjoint union of at most $r-1$ complete graphs.  Like $\varphi$, it was shown that $\varphi ^{(r)}$ lifted to complete subhypergraphs in a predictable way, but unfortunately, complements were no longer preserved when $r>3$ making it ineffective as a tool in Ramsey theory.  
Rather, \cite{BR} included an application involving Tur\'an numbers.  

Since $\varphi$ preserved complements, it could be interpreted as describing a way of lifting $2$-colorings of the edges of the complete graph of order $n\ge 3$ to $2$-colorings of the hyperedges of the complete $3$-uniform hypergraph of order $n$.  Generalizing this interpretation to $3$-colorings was considered in \cite{BHP}, but the problem of deciding how a rainbow triangle should lift led to a focus on Gallai colorings (those lacking rainbow triangles).  The present paper sets out to avoid this restriction by recognizing the lifting as a linear transformation between certain vector spaces over a finite field.  This generalization allows one to extend the lifting map to more than two colors and to consider it as a map between arbitrary uniformities.  It also provides an algebraic framework to the lifting map, providing insight into its structure via standard algebraic techniques.

In Section \ref{lineartrans}, we construct vector spaces of hypergraph edge colorings over finite fields and interpret the lifting map as a linear transformation between such vector spaces.  A few general results are proved before focusing our attention on the theory when the field of scalars is $\F _2$ (the finite field of order $2$) in Section \ref{twofield}.  Finally, in Section \ref{Ramtheory}, we consider the applications to Ramsey theory that follow from our new algebraic description of the lifting map.  We are able to prove two new lower bounds for certain $3$-color and $5$-color $3$-uniform hypergraph Ramsey numbers.

\section{The Lifting Map as a Linear Transformation}\label{lineartrans}

In order to establish the lifting map as a linear transformation, we must first formalize the terminology and background surrounding the objects to be studied.  An $r$-uniform hypergraph $H=(V(H), E(H))$ consists of a nonempty set of vertices $V(H)$ and a set of hyperedges $E(H)$, whose elements are different $r$-tuples of distinct vertices from $V(H)$.  The complete $r$-uniform hypergraph of order $n$ is denoted by $K_n^{(r)}$ and consists of $n$ vertices, every $r$-element subset of which forms a hyperedge.  When $r=2$, we simplify the notation $K_n^{(2)}$ and just write $K_n$.

Let $\F_q$ be the finite field of order $q=p^m$, where $p$ is a prime number and $m\ge 1$ is an integer.  An $\F_q$-hyperedge coloring of an $r$-uniform hypergraph $H$ is a map $f: E(H)\longrightarrow \F_q$.  Denote the set of all $\F_q$-hyperedge colorings of $K_n^{(r)}$ by  $\mathcal{H}_n^{(r)}(\F_q)$ and observe that it forms a vector space over $\F_q$ under the operations $$(f+g)(e)=f(e)+g(e) \quad \mbox{and} \quad (\alpha f)(e)=\alpha f(e),$$ where $e\in E(K_n^{(r)})$, $\alpha \in \F_q$, and $f,g\in \mathcal{H}_n^{(r)}(\F_q)$. A basis for $\mathcal{H}_n^{(r)}(\F_q)$ can be formed using the $\F_q$-hyperedge colorings $$f_
{e'}(e)=\left\{\begin{array}{ll} 1 & \mbox{if $e=e'$} \\ 0 & \mbox{if $e\ne e'$,} \end{array}\right.$$ where $e' \in E(K_n^{(r)})$. It follows that $dim_{\F_q}(\mathcal{H}_n^{(r)}(\F_q))={n\choose r}$.

If $T$ is a set, then  denote by $T^r$ the set of all $r$-element subsets of $T$.  For $2\le s<r\le n$, define the lifting $\Psi _{q,n}^{(s,r)}: \mathcal{H}_n^{(s)}(\F_q)\longrightarrow \mathcal{H}_n^{(r)}(\F_q)$ by $$(\Psi _{q,n}^{(s,r)} f)(e)=\mathop{\sum}\limits_{e'\in e^{s}} f(e'),$$ where $f\in \mathcal{H}_n^{(s)}(\F_q)$ and $e\in E(K_n^{(r)})$.  We leave it as an exercise for the reader to check that $\Psi_{q,n}^{(s,r)}$ is a linear transformation and $\Psi_{2,n}^{(2,3)}$ corresponds with the lifting described in \cite{BHLS}. Realizing this map as a linear transformation elucidates some of its properties, the first of which involves the time required to determine if a given hyperedge coloring is in the image of such a map.

\begin{proposition}\label{p:algorithm-exists}
For $g \in \mathcal{H}_n^{(r)}(\F_q)$, there exists a polynomial time algorithm to determine whether or not $g \in Im (\Psi_{q,n}^{(s,r)})$. 
\end{proposition}

\begin{proof}
Apply Gaussian Elimination to solve the system $\Psi_{q,n} ^{(s,r)}( f) = g$. 
\end{proof}

\begin{theorem} \label{sumfield}
Let $n\ge r>s\ge 2$, $g\in Im (\Psi_{q,n}^{(s,r)})$, and suppose that $q| {n-s \choose r-s}$.  Then   $$\mathop{\sum}_{e\in E(K_n^{(r)})} g(e)=0_{\F_q}.$$
\end{theorem}

\begin{proof}
Let $g=\Psi_{q,n}^{(s,r)}f$.
By definition, the sum in Theorem \ref{sumfield} becomes $$\mathop{\sum}_{e\in E(K_n^{(r)})} g(e)=\mathop{\sum}_{e\in E(K_n^{(r)})} \mathop{\sum}_{e'\in e^{s}} f(e').$$  Within this sum, $f(e')$ occurs ${n-s \choose r-s}$ times, corresponding to the number of hyperedges $e\in E(K_n^{(r)})$ that contain $e'$.  The assumption $q| {n-s \choose r-s}$ implies that the sum is $0_{\F_q}$.
\end{proof}

In the case where $s=r-1$, we obtain the following theorem.

\begin{theorem}\label{minnumber}
Let $r\ge 3$ and assume that $\Psi^{(r-1,r)}_{q,n}f=\Psi^{(r-1,r)}_{q,n}g$. If $f\ne g$, then $f$ and $g$ differ by at least $n-r+2$ hyperedge colors.
\end{theorem}

\begin{proof}
Assume that $\Psi^{(r-1,r)}_{q,n}f=\Psi^{(r-1,r)}_{q,n}g$ and $f\ne g$.  Then some hyperedge $x_1 x_2 \cdots x_{r-1}$ in $K_n^{(r-1)}$ receives a different color under $g$ than it does under $f$.  The hyperedge $x_1 x_2 \cdots x_{r-1}$ is contained in exactly $n-(r-1)$ $r$-tuples, each of which contains a distinct vertex from the set $$V(K_n^{(r)})-\{x_1, x_2, \dots , x_{r-1}\}=\{y_1, y_2, \dots , y_{n-(r-1)}\}.$$  Retaining the colors of these $r$-tuples under $\Psi^{(r-1,r)}_{q,n}$ requires at least $n-(r-1)$ additional hyperedges in $K_n^{(r-1)}$ (each of which includes a single element from $\{y_1, y_2, \dots , y_{n-(r-1)}\}$ and some selection of $r-2$ vertices from $\{x_1, x_2, \dots ,x_{r-1}\}$) be colored differently under $g$ than under $f$.  Hence, $f$ and $g$ differ by at least $n-r+2$ hyperedge colors.
\end{proof}

To see that this theorem is optimal, consider the case of $\Psi^{(2,3)}_{2,5}$.   By Theorem 4 of \cite{BHLS}, both a $1_{\F_2}$-colored $K_1\dot{\cup} K_4$ and a $1_{\F_2}$-colored $K_2\dot{\cup} K_3$ map to a $1_{\F_2}$-colored $K_5^{(3)}$ and can be shown to differ by exactly $4$ edge colors, the minimum number implied by Theorem \ref{minnumber}.

\section{The Case of the Finite Field $\F_2$}\label{twofield}

When restricting to the case $q=2$, we have the ability to discuss complements of hyperedge colorings.  The preservation of complements is exactly the property that allowed the lifting map to be applied to Ramsey theory in \cite{BHLS}. 
Let $H$ be an $r$-uniform hypergraph and let $f: E(H)\longrightarrow \F_2$ be an $\F_2$-hyperedge coloring. Define the {\it complement} of $f$ to be the $\F_2$-hyperedge coloring $\overline{f}:E(H)\longrightarrow \F_2$ such that for all $e \in E(H)$,
$$\overline{f}(e)=\left\{\begin{array}{ll} 0_{\F_2} & \mbox{if $f(e)=1_{\F_2}$} \\ 1_{\F_2} & \mbox{if $f(e)=0_{\F_2}$.} \end{array}\right.$$

\begin{theorem}
Let $f\in \mathcal{H}_n^{(s)}(\F_2)$. Then
$$\Psi _{2,n}^{(s,r)} \overline{f}=\left\{\begin{array}{ll} \overline{\Psi _{2,n}^{(s,r)}f } & \mbox{if ${r\choose s}$ is odd} \\ 
\Psi _{2,n}^{(s,r)} f & \mbox{if ${r\choose s}$ is even.}\end{array}\right.$$
\label{compthrm}\end{theorem}

\begin{proof}
Let $f \in \mathcal{H}_n^{(s)}(\F_2)$. Note that each $r$-uniform hyperedge $e\in E(K_n^{(r)})$ corresponds with a selection of $r$ vertices in $K_{n}^{(s)}$. Let $H$ be the subhypergraph induced by $e$ in $K_n^{(s)}$ and observe that $H$ has ${r\choose s}$ $s$-uniform hyperedges. Denote the subhypergraph of $H$ spanned by all $0_{\F_2}$-colored hyperedges in $f$ by $H_0$ and the subhypergraph of $H$ spanned by all $1_{\F_2}$-colored hyperedges in $f$ by $H_1$.  Define the subhypergraphs $\overline{H_0}$ and $\overline{H_1}$ similarly under $\overline{f}$ and note that $\overline{H_0} \cong H_1$ and $\overline{H_1} \cong H_0$. 
By definition, the $r$-uniform hyperedge $e$ will be $1_{\F_2}$-colored in $\Psi _{2,n}^{(s,r)} f$ if $|E(H_1)|$ is odd and $0_{\F_2}$-colored if $|E(H_1)|$ is even. 
In the case where ${r\choose s}$ is odd, exactly 1 of $|E(H_1)|$ and $|E(H_0)|$ must be odd. Without loss of generality, assume $|E(H_1)|$ is odd. This means that $|E(\overline{H_1})|$ is even and while the hyperedge  $e$ is $1_{\F_2}$-colored in $\Psi _{2,n}^{(s,r)} \overline{f}$, $e$ is $0_{\F_2}$-colored in $\Psi _{2,n}^{(s,r)} f$.  It follows that $$\Psi _{2,n}^{(s,r)} \overline{f} = \overline{\Psi _{2,n}^{(s,r)} f}$$ when ${r\choose s}$ is odd.  In the case where ${r\choose s}$ is even, either $|E(H_1)|$ and $|E(H_0)|$ are both odd or they are both even. Likewise, either $|E(H_1)|$ and $|E(\overline{H_1})|$ are both odd or they are both even. This means that the hyperedge $e$ receives the same color in $\Psi _{2,n}^{(s,r)} \overline{f}$ as it receives in $\Psi _{2,n}^{(s,r)} f$, proving that $$\Psi _{2,n}^{(s,r)} \overline{f} = \Psi _{2,n}^{(s,r)} f$$ whenever ${r\choose s}$ is even.
\end{proof}

Next, we consider the sum introduced in Theorem \ref{sumfield} in the case where $q=2$.  To simplify the statement of the next theorem, for $g\in \mathcal{H}_n^{(r)}(\F_2)$, write $$S^{(r)}_{2,n}(g):=\mathop{\sum}\limits_{e\in K_n^{(r)}} g(e).$$ 

\begin{theorem}
If $n\ge r>s\ge 2$ and $g= \Psi_{2,n}^{(s,r)}f$, then $$S^{(r)}_{2,n}(g)=\left\{ \begin{array}{ll} 0_{\F_2} & \mbox{if ${n-s \choose r-s}$ is even} \\ S^{(s)}_{2,n}(f) & \mbox{if ${n-s \choose r-s}$ is odd.}\end{array}\right.$$
\end{theorem}

\begin{proof}
The first case, where ${n-s \choose r-s}$ is even follows from Theorem \ref{sumfield}.  In the case where ${n-s \choose r-s}$ is odd, observe that $$S^{(r)}_{2,n}(g)=\mathop{\sum}\limits_{e\in K_n^{(r)}} \mathop{\sum}\limits_{e'\in e^s} f(e'),$$ with $f(e')$ occurring ${n-s\choose r-s}$ times.  It follows that the sum simplifies to $$\mathop{\sum}\limits_{e'\in E(K_n^{(s)})} f(e')$$ in this case.
\end{proof}

Now we consider a couple of theorems in the case where $s=r-1$.  The following theorem was motivated by Theorem 2.1 of \cite{BR}.

\begin{theorem}
Let $r\ge 3$ be odd and consider a coloring $g=\Psi_{2,n}^{(r-1, r)}f$. If the image of a subhypergraph $K$ of $K_n^{(r-1)}$ under $g$ is complete in some color, then $K$ consists of at most $r-1$ connected components in that color.
\end{theorem}

\begin{proof}
Suppose false, then $K$ consists of at least $r$ components.  Select a single vertex from $r$ of the components and consider the resulting hypergedge. In the case where the components are $1_{\F_2}$-colored, all hyperedges induced by these vertices will be $0_{\F_2}$-colored and the resultant hyperedge will be $0_{\F_2}$-colored. By our assumption that $K$ is monochromatic, all other resultant hyperedges must be colored $0_{\F_2}$ as well. Now consider selecting $r-1$ adjacent vertices from one of the components and a single vertex from one other. These vertices will induce a single edge $1_{\F_2}$-colored $r-1$-uniform hyperedge resulting in a $1_{\F_2}$-colored $r$-uniform hyperedge. This gives us a contradiction. In the case where the components are $0_{\F_2}$-colored, all edges induced by these vertices will be $1_{\F_2}$-colored. Since $r$ is odd, we have an odd number of $1_{\F_2}$-colored $r-1$-uniform hyperedges which will result in a $1_{\F_2}$-colored hyperedge. As before, select $r-1$ vertices joined by a $0_{\F_2}$-colored hyperedge and a single vertex from one other component. These vertices induce $r-1$ $1_{\F_2}$-colored hyperedges which, because $r-1$ is even, result in a $0_{\F_2}$-colored $r$-uniform hyperedge. Once again, this contradicts the assumption that $K$ is monochromatic, proving that $K$ consists of at most $r-1$ connected components.
\end{proof}

The following is a generalization of Theorem 7 of \cite{BHLS}.

\begin{theorem}
If $r\ge 3$ is odd, then $K_m^{(r)}-e$ never occurs as an induced subhypergraph in any hyperedge coloring $\Psi_{2,n}^{(r-1, r)} f$, where $n\ge r+1$.
\end{theorem}

\begin{proof}
Suppose that $K_m^{(r)}-e$ does occur as an induced subhypergraph in some hyperedge coloring $\Psi_{2,n}^{(r-1, r)} f$.  Then the coloring necessarily contains an induced subhypergraph isomorphic to $K_{r+1}^{(r)}-e$.   First, consider the case where the $K_{r+1}^{(r)}-e$ is in color $1_{\F_2}$ (so that $e$ has color $0_{\F_2}$).  There must necessarily be an even number of $(r-1)$-element subsets of $e$ that are colored $1_{\F_2}$ in $f$ and an odd number that are colored $0_{\F_2}$ (since $r$ is odd). For any other hyperedge in the $K_{r+1}^{(r)}-e$ other than $e$, there are an odd number of $(r-1)$-element subsets that are colored $1_{\F_2}$ and an even number that are colored $0_{\F_2}$.  Summing over all $r$-element subsets, we obtain an odd number of $(r-1)$-elements subsets in color $1_{\F_2}$.  However, note that each $(r-1)$-element subset occurs in exactly $2$ $r$-element subsets.  It follows that this sum should be $0_{\F_2}$, giving a contradiction.  The case where $K_{r+1}^{(r)}-e$ is in color $0_{\F_2}$ is the same, but with $0_{\F_2}$ and $1_{\F_2}$ switched.
\end{proof}

For the remainder of this section, we consider the lifting of graphs to $r$-uniform hypergraphs (the case $s=2$).  Our investigation focuses next on the graphs that lift to complete and empty subhypergraphs (in color $1_{\F_2}$), so to simplify the statements of our theorems, we introduce some new terminology.
Let $G$ be a graph of order $n$ and let $G'$ be the subgraph of $G$ induced by a selection of $r$ vertices such that $2 < r \leq n$.  If $G'$ always has an odd number of edges, $G$ is called 
{\it $r$-complete}. If $G'$ always has an even number of edges, $G$ is called {\it $r$-void}. If $G$ is neither $r$-complete nor $r$-void, it is called {\it $r$-neutral}.
Under $\Psi _{2,n}^{(2,r)}$, a $r$-complete graph in color $1_{\F_2}$ will always lift to a complete $1_{\F_2}$-colored $r$-uniform hypergraph, while an $r$-void graph will always lift to a complete $0_{\F_2}$-colored $r$-uniform hypergraph.

\begin{theorem}
Given an odd $r\ge 3$, a complete bipartite graph of order at least $r$ is always $r$-void.
\label{bipartvoidthrm}\end{theorem}

\begin{proof}
Let $r\ge 3$ be odd and $n\ge r$ such that $K_{s,t}$ is a complete bipartite graph of order $n$.  The vertex set of $K_{s,t}$ is the disjoint union of partite sets $V_1$ and $V_2$, having cardinalities $|V_1|=s$ and $|V_2|=t$, where $n=s+t$.  Any selection of $r$ vertices from $K_{s,t}$ results in $u$ vertices selected from $V_1$ and $r-u$ vertices selected from $V_2$.  Since $r$ is odd, exactly one of $u$ and $r-u$ must be odd. Without loss of generality, let $u$ be odd and $r-u$ be even. Since $K_{s,t}$ is complete bipartite, each of the $u$ vertices selected from $V_1$ must share an edge with each of the $r-u$ vertices selected from $V_2$. This means that the subgraph induced by these vertices must have $u \cdot (r-u)$ edges. As $r-u$ is even, $u \cdot (r-u)$ must also be even.
\end{proof}

\begin{lemma}
Let $r>u$.
\begin{enumerate}
    \item If $r \equiv 0 \Mod 4$, then ${u\choose 2} + {r-u\choose 2} \equiv u \Mod{2}$.
    \item If $r \equiv 2 \Mod 4$, then ${u\choose 2} + {r-u\choose 2} \equiv (u+1) \Mod{2}$.
\end{enumerate}
\label{lem1}\end{lemma}

\begin{proof} We run through cases, based on the value of $r$ modulo $4$.\\
Case 1: Suppose that $r \equiv 0 \Mod 4$.  When $u$ is even (i.e. $u\equiv 0,2 \Mod 4$), then ${u\choose 2}$ and ${r-u\choose 2}$ have the same parity, implying that  ${u\choose 2} + {r-u\choose 2}$ is even.  When $u$ is odd (i.e., $u\equiv 1,3 \Mod 4$), then ${u\choose 2}$ and ${r-u\choose 2}$ have different parities, implying that  ${u\choose 2} + {r-u\choose 2}$ is odd. \\
Case 2: Suppose that $r \equiv 2 \Mod 4$.  When $u$ is even (i.e., $u\equiv 0,2 \Mod 4$), then ${u\choose 2}$ and ${r-u\choose 2}$ have different parities, implying that  ${u\choose 2} + {r-u\choose 2}$ is odd.  When $u$ is odd (i.e., $u\equiv 1,3 \Mod 4$), then ${u\choose 2}$ and ${r-u\choose 2}$ have the same parity, implying that  ${u\choose 2} + {r-u\choose 2}$ is even.
\end{proof}

\begin{theorem}
Let $G$ be the disjoint union of 2 complete subgraphs of orders $s$ and $t$ with $r < s + t$. $G$ is $r$-neutral if $r$ is even, $G$ is $r$-void if $r \equiv 1 \Mod{4}$, and $G$ is $r$-complete if $r \equiv 3 \Mod{4}$. 
\end{theorem}

\begin{proof}
First, consider the case where $r$ is even. We select $u<r$ vertices from the $K_s$ and $r-u$ vertices from the $K_t$. The subgraph induced by these $r$ vertices will contain ${u\choose 2} + {r-u\choose 2}$ edges. Next, select $u+1$ vertices from the $K_s$ and $r-u-1$ vertices from the $K_t$. The subgraph induced by these $r$ vertices will contain ${u+1\choose 2} + {r-u-1\choose 2}$ edges. By Lemma \ref{lem1}, these two selections will yield differing numbers of edges modulo 2. Therefore $G$ must be $r$-neutral. Now consider the case where $r$ is odd. By Theorem \ref{bipartvoidthrm}, $\overline{G}$ is $r$-void, so by Theorem \ref{compthrm}, if $r \equiv 1 \Mod{4}$, $G$ is $r$-void and if $r \equiv 3 \Mod{4}$, $G$ is $r$-complete.
\end{proof}

We conclude this section with a result concerning the original lifting map $\Psi_{2,n}^{(2,3)}
$, denoted by $\varphi$ in \cite{BHLS}.

\begin{theorem}\label{preimage}
The lifting $\Psi_{2,n} ^{(2,3)}$ is a $2^{n-1}$-to-one mapping.
\end{theorem}

\begin{proof}
The fact that every $g\in Im (\Psi_{2,n} ^{(2,3)} )$ has the same number $k$ of preimages follows from the linearity of $\Psi_{2,n} ^{(2,3)}$.  In order to determine the value of $k$, we select an element in the image whose preimages we can easily count.  Specifically, consider the element in $\mathcal{H}^{(3)}_{n}(\F_2)$ that maps all hyperedges to $0_{\F_2}$.  The induced subhypergraph in color $0_{\F_2}$ is isomorphic to $K_n^{(3)}$, and by Theorem 4 of \cite{BHLS}, the elements in $\mathcal{H}^{(2)}_n(\F_2)$ that map to this element are those in which color $0_{\F_2}$ is given to subgraphs that are complete of order $n$ or are the disjoint union of two complete subgraphs, whose orders add to $n$.  When $n$ is even, the possibilities are $$K_n, \ K_1 \dot{\cup} K_{n-1}, \ K_2 \dot{\cup} K_{n-2} , \ \dots , \ K_{n/2-1} \dot{\cup} K_{n/2+1}, \ K_{n/2} \dot{\cup} K_{n/2},$$ with these cases occurring $${n\choose 0}, \ {n\choose 1} , \ {n\choose 2}, \ \dots , {n\choose n/2-1}, \ \frac{1}{2}{n\choose n/2}$$ times, respectively.  When $n$ is odd, the possibilities are $$K_n, \ K_1 \dot{\cup} K_{n-1}, \ K_2 \dot{\cup} K_{n-2} , \ \dots , \ K_{(n-1)/2} \dot{\cup} K_{(n+1)/2},$$ with these cases occurring $${n\choose 0}, \ {n\choose 1} , \ {n\choose 2}, \ \dots , \ {n\choose (n-1)/2}$$ times, respectively.  Applying the identity $${n\choose 0} + {n\choose 1}+ \cdots + {n\choose n} =2^n, $$ along with the property ${ n\choose k}={n\choose n-k}$ to each of these cases, we obtain the statement of the theorem.
\end{proof}

\section{Some Applications to Ramsey Theory}\label{Ramtheory}

The generalization of the lifting map as a linear transformation allows us to prove a $3$-colored Ramsey number bound that does not require the restriction to Gallai colorings, as in \cite{BHP}.  Recall that if $H_1, H_2, \dots , H_t$ are $r$-uniform hypergraphs, then the Ramsey number $R(H_1, H_2, \dots , H_t;r)$ is the least positive integer $p$ such that every $t$-coloring of the hyperedges of $K_p^{(r)}$ results in a subhypergraph isomorphic to $H_i$ spanned by hyperedges in color $i$, for some $1\le i\le t$. 

\begin{theorem}\label{ram} Let $s_i\ge 3$ for all $1\le i\le 3$.
$$R(K_{2s_1-1}^{(3)}-e, K_{2s_2-1}^{(3)},  K_{2s_3-1}^{(3)};3)\ge R(K_{s_1}, K_{s_2},  K_{s_3};2).$$
\end{theorem}

\begin{proof} We identify the three colors with the elements in $\F_3$.
This proof follows the proof of the analogous Gallai-Ramsey number result in Theorem 3 of \cite{BHP}.  For all $3$-tuples of vertices in $K_n$ whose induced subgraphs are monochromatic or $2$-colored, the image is the expected image in the Gallai-Ramsey case.  For rainbow $3$-tuples, a single color results, corresponding to the $2$-color lifting where the other two colors are identified as being the same. By Theorem 4 of \cite{BHLS}, avoiding a monochromatic $K_{s_i}$ in color $i$ results in an image that avoids a monochromatic copy of $K_{2s_i-1}^{(3)}$ in color $i$.  For color $0_{\F_3}$ in particular, the lifting is the same as in the $2$-color case.  This allows us to further note that no $0_{\F_3}$-colored copy of $K_{2s_1-1}^{(3)}-e$ exists by Theorem 7 of \cite{BHLS}.  
\end{proof}

The following theorem generalizes Theorem 10 of \cite{BHLS} to a $5$-color Ramsey-theoretic result.

\begin{theorem} \label{ram2}
Let $q\ge 3$ and $s_i\ge 3$ for all $1\le i\le 3$. Then $$R(K_{2s_1-1}^{(3)}-e, K_{2s_2-1}^{(3)}, K_{2s_3-1}^{(3)}, K_5^{(3)}, K_{q+1}^{(3)}-e;3)>q(R(K_{s_1}, K_{s_2}, K_{s_3};2)-1).$$
\end{theorem}

\begin{proof}
If $p=R(K_{s_1}, K_{s_2}, K_{s_3};2)$, then by Theorem \ref{ram}, we can construct a $3$-colored $K_{p-1}^{(3)}$ that avoids a copy of $K_{2s_1-1}^{(3)}-e$ in color $0_{\F_3}$, a copy of $K_{2s_2-1}^{(3)}$ in color $1_{\F_3}$, and a copy of $K_{2s_3-1}^{(3)}$ in color $2_{\F_3}$.  Consider the disjoint union of $q$ copies of this $3$-colored $K_{p-1}^{(3)}$, and label the copies $V_1, V_2, \dots ,V_q$.  Color the hyperedges that have one vertex in some $V_i$ and the other two vertices in some $V_j$ ($i\ne j$) using a fourth color and color the hyperedges that have all vertices coming from distinct $V_i$ using a fifth color.  It is easily confirmed that the largest complete hypergraph in the fifth color uses at most one vertex from any given $V_i$ and adding in another vertex to such a hypergraph is lacking more than one hyperedge in the fifth color (hence, avoiding a $K_{q+1}^{(3)}-e$). The largest complete hypergraph in the fourth color uses at most two vertices from at most two distinct $V_i$.  It follows that no $K_5^{(3)}$  exists in the fourth color and no $K_{q+1}^{(3)}-e$ exists in the fifth color.   
\end{proof}

Using the $3$-color lower bounds given in Table X of Section 6.1 of Radziszowski's dynamic survey \cite{Rad}, Theorem \ref{ram2} implies the following lower bounds.  Here, we use the usual notation $R_3(G)$ to denote the $3$-color graph Ramsey number $R(G,G,G;2)$.

\begin{align}
    R_3(K_3)=17 \quad &\Longrightarrow \quad R(K_{5}^{(3)}-e, K_5^{(3)}, K_5^{(3)}, K_5^{(3)}, K_{q+1}^{(3)}-e;3)>16q, \notag \\
    R_3(K_4)\ge 128 \quad &\Longrightarrow \quad R(K_7^{(3)}-e, K_7^{(3)}, K_7^{(3)}, K_5^{(3)}, K_{q+1}^{(3)}-e;3)>127q,\notag \\
    R_3(K_5)\ge 417 \quad &\Longrightarrow \quad R(K_9^{(3)}-e, K_9^{(3)}, K_9^{(3)}, K_5^{(3)}, K_{q+1}^{(3)}-e;3)>416q, \notag \\
    R_3(K_6)\ge 1070 \quad &\Longrightarrow \quad R(K_{11}^{(3)}-e, K_{11}^{(3)}, K_{11}^{(3)}, K_5^{(3)}, K_{q+1}^{(3)}-e;3)>1069q, \notag \\
    R_3(K_7)\ge 3214 \quad &\Longrightarrow \quad R(K_{13}^{(3)}-e, K_{13}^{(3)}, K_{13}^{(3)}, K_5^{(3)}, K_{q+1}^{(3)}-e;3)>3213q, \notag \\
    R_3(K_8)\ge 6079 \quad &\Longrightarrow \quad R(K_{15}^{(3)}-e, K_{15}^{(3)}, K_{15}^{(3)}, K_5^{(3)}, K_{q+1}^{(3)}-e;3)>6078q, \notag \\
    R_3(K_9)\ge 13761 \quad &\Longrightarrow \quad R(K_{17}^{(3)}-e, K_{17}^{(3)}, K_{17}^{(3)}, K_5^{(3)}, K_{q+1}^{(3)}-e;3)>13760q. \notag 
\end{align}

\bibliographystyle{amsplain}

\begin{thebibliography}{10}

\bibitem{BHLS} M. Budden, J. Hiller, J. Lambert, and C. Sanford, {\it The Lifting of Graphs to $3$-Uniform Hypergraphs and Some Applications to Hypergraph Ramsey Theory}, Involve {\bf 10} (2017), 65-76.

\bibitem{BHP} M. Budden, J. Hiller, and A. Penland, {\it Constructive Methods in Gallai-Ramsey Theory for Hypergraphs,} Integers {\bf 20A} (2020), \#A4.

\bibitem{BR} M. Budden and A. Rapp, {\it Constructing $r$-Uniform Hypergraphs with Restricted Clique Numbers,} The North Carolina Journal of Mathematics and Statistics {\bf 1} (2015), 30-34.

\bibitem{Rad} S. Radziszowski, {\it Small Ramsey Numbers - Revision 15} Electron. J. Combin.  {\bf DS1.15}  (2017), 1-104.

\end{thebibliography}

\end{document}